\documentclass[oribibl]{llncs}
\usepackage{graphicx}
\usepackage[psamsfonts]{amssymb}
\usepackage{bm,amsmath,amsfonts}
\usepackage{amstext}

\newcommand{\so}{\mathrm{o}}
\newcommand{\lo}{\mathrm{O}}
\newcommand{\ep}{\epsilon}
\newcommand{\de}{\delta}
\newcommand{\si}{\sigma}
\newcommand{\E}{\mathrm{E}}
\newcommand{\dmin}{D_{\mathrm{min}}}
\newcommand{\n}{\nonumber}
\newcommand{\nn}{\nonumber\\}

\newcommand{\fhat}{\hat{F}}
\newcommand{\muhat}{\hat{\mu}}
\newcommand{\fhatn}[2]{\fhat_{#1}(#2)}
\newcommand{\muhatn}[2]{\muhat_{#1}(#2)}
\newcommand{\rd}{\mathrm{d}}
\newcommand{\dmo}[1]{\bm{M}^{(#1)}}

\newcommand{\jni}[2]{J_{#2}(#1)}
\newcommand{\mf}{\mathcal{F}}
\newcommand{\mom}{\frac{1}{1-\mu}}
\newcommand{\momss}{(1-\mu)^{-1}}
\newcommand{\mo}[1]{\bm{M}}
\newcommand{\moo}[1]{\tilde{\bm{M}}}
\newcommand{\moap}[1]{\tilde{\bm{M}}}
\newcommand{\moapo}{M}
\newcommand{\omo}[2]{\E^{(#1)}(#2)}
\newcommand{\dminm}[1]{\dmin^{(#1)}}
\newcommand{\bibun}[2]{\frac{\rd #1}{\rd #2}}
\newcommand{\mohat}[3]{\omo{#1}{\fhatn{#2}{#3}}}
\newcommand{\mass}[1]{\de\left(#1\right)}
\newcommand{\vs}{\mathcal{V}}
\newcommand{\com}{\enspace ,}
\newcommand{\per}{\enspace .}
\newcommand{\mm}{\mathcal{M}}
\newcommand{\uF}{\underline{F}}
\newcommand{\bF}{\bar{F}}
\newcommand{\dmint}[1]{\tilde{D}^{(#1)}_{\mathrm{min}}}
\newcommand{\ta}{p}
\newcommand{\tb}{q}
\newcommand{\mfm}[1]{\mf(#1)}

\newcommand{\bG}{\bar{G}}

\title{Stochastic Bandit Based on Empirical Moments}

\author{Junya Honda\inst{1}\thanks{Supported by JSPS Research Fellowships for Young Scientists.}
and Akimichi Takemura\inst{2}\thanks{Supported by Aihara Project, the FIRST program from JSPS.}}
\institute{
Graduate school of Frontier Sciences,
\and
Graduate school of Information Science and Technology,\\
The University of Tokyo, Japan.\\
\email{\small \{honda,takemura\}@stat.t.u-tokyo.ac.jp}}
\begin{document}
\maketitle

\begin{abstract}
In the multiarmed bandit problem a gambler chooses an arm of a slot machine to pull
considering a tradeoff between exploration and exploitation.
We study the stochastic bandit problem where each arm has a reward
distribution supported in a known bounded interval, e.g. $[0,1]$.
For this model, 	policies which take into account the empirical variances
(i.e. second moments) of the arms are known to perform effectively.
In this paper, we generalize this idea
and we propose a policy which exploits the first $d$ empirical moments for arbitrary $d$
fixed in advance. 
The asymptotic upper bound of the regret of the policy approaches
the theoretical bound
by Burnetas and Katehakis as $d$ increases.
By choosing appropriate $d$, the proposed policy realizes a tradeoff between
the computational complexity and the expected regret.
\end{abstract}

\section{Introduction}
\label{section-intro}
The multiarmed bandit problem is one of the formulations of the tradeoff
between exploration and exploitation.
This problem is based on an analogy with a gambler playing a slot machine with
more than one arm.
The gambler pulls arms sequentially so that the total reward is maximized.

We consider a $K$-armed stochastic bandit problem originally considered in \cite{lai}.
There are $K$ arms and each arm $i=1,\cdots,K$
 has a probability distribution $F_i$ with the expected value
$\mu_i$.
The gambler chooses an arm to pull based on a policy and
 receives a reward according to $F_i$ independently in
each round.
We call an arm $i$ optimal if $\mu_i=\mu^*$ and suboptimal if $\mu_i<\mu^*$.
Then, the goal of the gambler is to maximize the sum of the rewards
by pulling optimal arms
as often as possible.
Many researches have been conducted for the stochastic bandit problem
\cite{conti1,even,meuleau,strens,vermorel,gittins}
as well as the non-stochastic bandit \cite{adversarial, uchiya}.

In this paper we consider the model $\mf$, the family of distributions with supports contained
in the bounded interval $[0,1]$. 
The gambler knows that each distribution $F_i$ is included in $\mf$.
For this model Upper Confidence Bound (UCB) policies are popular for their simple
 form and fine performance \cite{ucb,ucbv}.
%
Recently Honda and Takemura \cite{honda_colt} proposed
Deterministic Minimum Empirical \allowbreak Divergence (DMED) policy which satisfies
for arbitrary suboptimal arm $i$ that
\begin{eqnarray}
\E[T_i(n)] \le \left(\frac{1}{\dmin(F_i,\mu^*)}+\so(1)\right)\log n \label{ti_opt}
\end{eqnarray}
where $T_i(n)$ denotes the number of times that arm $i$ has been pulled over the first $n$ round
and
\begin{eqnarray}
\dmin(F,\mu)\equiv \min_{G\in\mf: \E_G[X]\ge \mu}D(F\Vert G)\n
\end{eqnarray}
with Kullback-Leibler divergence $D(\cdot\Vert\cdot)$.
DMED is asymptotically optimal since
the coefficient of $\log n$ on the right-hand side of \eqref{ti_opt} coincides with the theoretical bound given in \cite{burnetas}.
However, the complexity of the DMED policy is still larger than e.g. UCB policies,
although the computation involved in DMED is formulated
as a univariate convex optimization.
It is mainly because DMED requires the empirical distributions of the arms themselves
whereas other popular policies can be computed by the moments of the empirical distributions
of the arms, such as means and variances.

Now, our question is how we can bring the performance
close to the right-hand side of \eqref{ti_opt} 
by a policy which only considers the first $d$ empirical moments of the arms at each round.
In this paper, we propose {\it DMED-M} policy which is a variant of DMED and
is computable only by the empirical \underline{m}oments of the arms.
For arbitrary suboptimal arm $i$,
DMED-M satisfies
\begin{eqnarray}
\E[T_i(n)]\le \left(\frac{1}{\inf_{F\in\mf: \omo{d}{F}=\omo{d}{F_i}}\dmin(F,\mu^*)}
+\so(1)\right)\log n\com \label{ti_conj}
\end{eqnarray}
where $\omo{d}{F}\equiv(\E_{F_i}[X],\cdots,\E_{F_i}[X^d])$ denotes the
first $d$ moments of $F$ and this upper bound approaches \eqref{ti_opt} as $d\to\infty$.

DMED-M is obtained by an analogy with DMED.
Intuitively, DMED exploits the fact that the maximum likelihood that the arm with
empirical distribution $F_i$ is actually the best
is roughly $\exp(-t \dmin(F_i,\mu^*))$ for number of samples $t$.
When ignoring properties of the distribution $F_i$ except for
its first $d$ moments, we overestimate the maximum likelihood as
\begin{eqnarray}
\exp\Big(-t \inf_{
F\in\mf: \omo{d}{F}=\omo{d}{F_i}}\dmin(F,\mu^*)\Big)\n
\end{eqnarray}
instead of $\exp(-t \dmin(F_i,\mu^*))$
and the bound \eqref{ti_conj} appears correspondingly.

In DMED-M, it is necessary to compute
$\inf_{F\in\mf: \omo{d}{F}=(M_1,\cdots,M_d)}\allowbreak \dmin(F,\mu)$ for each round.
Classical results on {\it Tchebysheff systems} and {\it moment spaces} reveal that
$\bF$ attaining the infimum is
determined only by the value of the first $d$ moments $(M_1,\cdots,M_d)$
when the objective function $\dmin(\,\cdot\,,\mu)$ is included in a particular class.
Therefore the infimum is obtained by computing firstly the optimal solution
$\bF$ and then the value of the function $\dmin(\bF,\mu)$.
Both are obtained by solving polynomial equations and
DMED-M can be computed efficiently for small $d$.

This paper is organized as follows.
In Sect.\,\ref{section-preliminary}, we give definitions used
throughout this paper.
We propose DMED-M policy in Sect.\,\ref{section-proposed-policy}.
In Sect.\,\ref{section-dmin}, we study the minimization of $\dmin$ over distributions
whose first $d$ moments are common
for a practical implementation of DMED-M.
Proofs of results in Sects.\,\ref{section-proposed-policy} and \ref{section-dmin}
are given in Sect.\,\ref{sec_proofs}.
In Sect.\,\ref{section_gap}, we discuss an improvement of DMED-M
in terms of the worst case performance.
We present some simulation results on DMED-M in Sect.\,\ref{sect_sim}.
We conclude the paper with some remarks in Sect.\,\ref{section-remarks}.

\section{Preliminaries}\label{section-preliminary}
Let $\mf$ be the family of probability distributions on $[0,1]$
and $F_i\in\mf$ be the distribution of the arm $i=1,\dots,K$.
$\E_F[\cdot]$ denotes the expectation under $F\in\mf$.
When we write e.g. 
$\E_F[u(X)]$ for a function $u:\bbbr\to\bbbr$,
$X$ denotes a random variable with distribution $F$.
A set of probability distributions for $K$ arms is denoted by
$\bm{F}\equiv (F_1,\dots,F_K)\in \mathcal{F}^K\equiv\prod_{i=1}^K
\mathcal{F}$.
The expected value of arm $i$ is denoted by $\mu_i\equiv\E_{F_i}[X]$
and the optimal expected value is denoted by $\mu^*\equiv \max_{i}\mu_i$.

Let $T_i(n)$ be the number of times that arm $i$ has been pulled through the first $n$ rounds.
$\fhatn{i}{n}$ and $\muhatn{i}{n}$ denote the empirical distribution and the mean of arm $i$
after the first $n$ rounds, respectively.
$\hat{\mu}^*(n)\equiv \max_i\muhatn{i}{n}$ denotes the highest
empirical mean after the first $n$ rounds.
We call an arm $i$ a current best if $\muhatn{i}{n}=\hat\mu^*(n)$.

Now we review results in \cite{honda_colt}.
Define an index for $F\in \mathcal{F}$ and $\mu\in [0,1]$
\begin{eqnarray}
\dmin(F,\mu)\equiv \min_{G\in\mathcal{F}:\E(G)\ge \mu}D(F\Vert G)\com\n
\end{eqnarray}
where Kullback-Leibler divergence $D(F\Vert G)$ is given by
\begin{eqnarray}
D(F\Vert G)\equiv \begin{cases}
\E_F\left[\log
	       \frac{\mathrm{d}F}{\mathrm{d}G}\right]&\frac{\mathrm{d}F}{\mathrm{d}G}
\mbox{ exists,}\\
+\infty&\mbox{otherwise.}
\end{cases}\n
\end{eqnarray}
Under DMED policy proposed in \cite{honda_colt},
the expectation of $T_i(n)$ for any suboptimal arm $i$
is bounded as
\begin{eqnarray}
\E_{\bm{F}}[T_i(n)]\le \frac{1+\ep}{\dmin(F_i,\mu^*)}\log n
+\lo(1) \label{regret_dmed}
\end{eqnarray}
where $\ep>0$ is arbitrary.
The coefficient of the logarithmic term $1/\dmin(F_i,\mu^*)$ is the best possible \cite{burnetas}
and the following property holds for the function $\dmin(F,\mu)$.
\begin{proposition}[{\cite[Theorems 5 and 8]{honda_colt}}]
If $\E_F[X]\ge \mu$ then $\dmin(F,\mu)=0$.
If $\E_F[X]<\mu=1$ then $\dmin(F,\mu)=\infty$.
If $\E_F[X]<\mu<1$,
\begin{eqnarray}
\dmin(F,\mu)
&=&
\max_{0\le\nu\le\mom}\E_F[\log(1-(X-\mu)\nu)]\nn
&=&
\begin{cases}
\E_F\left[\log\left(1-X\right)\right]-\log(1-\mu)&\E_F\left[\frac{1}{1-X}\right]\le \frac{1}{1-\mu},\\
\max_{0<\nu<\mom}\E_F[\log(1-(X-\mu)\nu)]&\mathrm{otherwise},
\end{cases}\n
\end{eqnarray}
where we define $\log 0\equiv -\infty$.
\end{proposition}

Let $\omo{d}{F}\equiv (\E_F[X],\cdots,\E_F[X^d])$ denote the first $d$ moments of $F$.
The set of distributions with the first $d$ moments equal to $\mo{d}=(M_1,\cdots,M_d)$ is defined as
$\mfm{\mo{d}}\equiv\{F\in \mf: \omo{d}{F}=\mo{d} \}$.
We sometimes write $\dmo{d}$ instead of $\mo{d}$ to clarify the length of the vector.

Now define $\dminm{d}(\mo{d},\mu)$ by
\begin{eqnarray}
\dminm{d}(\mo{d},\mu)
&\equiv&
\inf_{F\in\mfm{\mo{d}}}\dmin(F,\mu)
\per\n
\end{eqnarray}
This function $\dminm{d}$ plays a central role throughout this paper.

%
%

\section{DMED-M Policy}\label{section-proposed-policy}
In this section we introduce DMED-M policy.
This policy determines an arm to pull based on the empirical moments of the arms.
DMED-M requires computation of the function $\dminm{d}$ and
we analyze this function in the next section.

In the following algorithm, each arm is pulled at most once in one loop.
Through the loop, the list of arms pulled in the next loop is determined.
$L_C$ denotes the list of arms to be pulled in the current loop.
$L_N$ denotes the list of arms to be pulled in the next loop.
$L_R\subset L_C$ denotes the list of remaining arms of $L_C$ which have
not yet been pulled in the current loop. 
The criterion for choosing an arm $i$ is the occurrence of the event $\jni{n}{i}$ given by
\begin{eqnarray}
\jni{n}{i}\equiv\{T_i(n)\dminm{d}(\mohat{d}{i}{n},\muhat^*(n))
\le \log n-\log T_i(n)\},\label{def-jn}
\end{eqnarray}
where $\omo{d}{\fhatn{i}{n}}$ represents the first $d$ empirical moments of arm $i$.

\begin{quote}
{\bf [DMED-M Policy]}

{\bf Parameter.}
Integer $d>0$.

{\bf Initialization.}
 $L_C,L_R:=\{1,\cdots,K\},\,L_N:=\emptyset$.
 Pull each arm  once.
$n:=K$.

{\bf Loop.}
\begin{itemize}
\item[1.] For $i\in L_C$ in the ascending order,
\begin{itemize}
\item[1.1.] $n:=n+1$ and pull arm $i$. $L_R:=L_R\setminus\{i\}$.
\item[1.2.] $L_N:=L_N\cup \{j\}$ (without a duplicate) for all
$j\notin L_R$
 such that $\jni{n}{j}$ occurs.
\end{itemize}
 \item[2.] $L_C,L_R:=L_N$ and $L_N:=\emptyset$.
\end{itemize}
\end{quote}
As shown above, $|L_C|$ arms are pulled in one loop.
At every round, arm $i$ is added to $L_N$ if $\jni{n}{i}$ occurs unless
$i\in L_R$, that is, arm $i$ is planned to be pulled in the remaining rounds in the current loop.
Note that if arm $i$ is a current best for the $n$-th round then
 $\jni{n}{i}$ holds since $\dmin(\mohat{m}{i}{n},\muhat^*(n))=0$ for this case.
Then $L_C$ is never empty.
Note that DMED in \cite{honda_colt} is obtained by replacing
$\dminm{d}(\mohat{d}{i}{n},\muhat^*(n))$ in \eqref{def-jn} by $\dmin(\fhatn{i}{n},\mu)$.
In view of Theorem \ref{thm_approach} below, 
DMED can be regarded as DMED-M with $d=\infty$.

\begin{theorem}\label{optMD}
Fix $\bm{F}\in\mf^K$ for which there exists a unique optimal arm $j$.
Under DMED-M policy, for any suboptimal arm $i$ and $\ep>0$ it holds that
\begin{eqnarray}
\E_{\bm{F}}[T_i(n)]\le \frac{1+\ep}{\dminm{d}(\omo{d}{F_i},\mu^*)}\log n
+\lo(1)\n
\end{eqnarray}
where $\lo(1)$ denotes a constant
dependent on $\ep$ and $\bm{F}$ but independent of $n$.
\end{theorem}
This theorem can be proved in a similar way as Theorem 4 of \cite{honda_colt} with
the fact that $\dmin(F,\mu)\allowbreak \ge \dminm{d}(\omo{d}{F},\mu)$ always holds.
However, we omit the proof
because it is long and very similar to 
the proof of Theorem 4 of \cite{honda_colt}.
The bound in Theorem \ref{optMD} approaches that of DMED given by \eqref{regret_dmed}
as $d\to\infty$ from the following theorem,
which we show in Sect.\,\ref{sec_proofs}.

\begin{theorem}\label{thm_approach}
For arbitrary $F\in\mf$ it holds that
\begin{eqnarray}
\lim_{d\to\infty}\dminm{d}(\omo{d}{F},\mu)=\dmin(F,\mu)\per\n
\end{eqnarray}
\end{theorem}

\section{Practical Representation of $\dminm{d}$}
\label{section-dmin}
For a computation and a theoretical evaluation of DMED,
it is essential to analyze the function
$\dminm{d}(\mo{d},\mu)=\inf_{F\in\mf: \omo{d}{F}=\mo{d}}\dmin(F,\mu)$.
In this section we study an explicit representation of this function.

The following theorem is the main result of the paper.
In this theorem, we identify a pair $(\{x_i\},\,\{f_i\})$ with a discrete
distribution such that $F(\{x_i\})=f_i$.
\begin{theorem}\label{thm_dminm}
If $\mfm{\mo{d}}=\{F\in\mf: \omo{d}{F}=\mo{d}\}$ is nonempty then
there exists
a unique optimal solution $\bar{F}\in \mf$ such that
\begin{eqnarray}
\dminm{d}(\mo{d},\mu)
=\inf_{F\in \mfm{\mo{d}}} \dmin(F,\mu) \label{thm_uniq}
=\dmin(\bar{F},\mu)\per
\end{eqnarray}
Furthermore, $F\in\mf$ is the unique optimal solution $\bar{F}$ if and only if
$\big((x_1,\cdots,x_l),\allowbreak (f_1,\cdots,f_l)\big)$ for $l=\lceil d/2 \rceil+1$ 
is a solution of
\begin{eqnarray}
\begin{cases}
\sum_{i=1}^{l} f_i x_i^m=M_m\;(m=0,\cdots,d),\, x_1=0,\,x_{l}=1,&\mbox{\rm $d$ is odd},\\
\sum_{i=1}^{l}f_i x_i^m=M_m\;(m=0,\cdots,d),\, \phantom{x_1=0,\,}\,x_{l}=1, &\mbox{\rm $d$ is even},
\end{cases}
 \label{thm_ifonly}
\end{eqnarray}
where we define the zeroth moment as $M_0=1$.
\end{theorem}
Note that the above $\bF$ only depends on the moment $\mo{d}$.
Then, the value of $\dminm{d}(\mo{d},\mu)$ is obtained by computing $\bF$ first
and then $\dmin(\bF,\mu)$.
Recall that $\dmin(\bar{F},\mu)=\max_{0\le\nu\le\momss}\E_{\bar{F}}[\log (1-(X-\mu)\nu)]$.
Since $\bar{F}$ has finite support $\{x_1,\cdots,x_l\}$,
the optimal solution $\nu^*$ attaining the maximum is
one of the boundary points $0,\,\momss$ or
an interior point $\nu\in(0,\momss)$ such that
\begin{eqnarray}
\bibun{}{\nu}\E_{\bF}[\log (1-(X-\mu)\nu)]=
\frac{\sum_{i=1}^l (\mu-x_i)\prod_{j\neq i}(1-(x_j-\mu)\nu)}{\prod_{i=1}^l(1-(x_i-\mu)\nu)}
=0\com \label{bibun0}
\end{eqnarray}
which is obtained by solving the $l$-th degree polynomial equation.
We give an explicit form of $\dminm{d}(\mo{d},\mu)$ for $d=1,2,3$
in the following theorem.

\begin{theorem}\label{thm_d123}
If $M_1<\mu<1$ and $\mfm{\mo{d}}$ is nonempty, then
$\dminm{d}(\mo{d},\mu)$ is expressed for $d=1,2$ as
\begin{eqnarray}
\dminm{1}(\mo{1},\mu)&=&
(1-M_1)\log\frac{1-M_1}{1-\mu}+M_1\log\frac{M_1}{\mu}\com \nn
\dminm{2}(\mo{2},\mu)&=&
\frac{(1-M_1)^2}{1-2M_1+M_2}\log \left(1-\left(\frac{M_1-M_2}{1-M_1}-\mu\right)\nu^{(2)}\right)\nn
&&+\frac{M_2-M_1^2}{1-2M_1+M_2}\log \left(1-\left(1-\mu\right)\nu^{(2)}\right)\n
\end{eqnarray}
where
\begin{eqnarray}
\nu^{(2)}=\frac{(1-M_1)(M_1-\mu)}{(1-M_1)\mu^2-(1-M_2)\mu+M_1-M_2}\per\n
\end{eqnarray}
For $d=3$ it is expressed as
\begin{eqnarray}
\dminm{3}(\mo{3},\mu)
&=&
\begin{cases}
\dminm{1}(\mo{1},\mu),&M_1=M_2=M_3,\\
\sum_{l=1}^3 f_l \log(1-\bar{x}_l\nu^{(3)})&\mathrm{otherwise},
\end{cases}\n
\end{eqnarray}
where
\begin{align}
&(\bar{x}_1,\bar{x}_2,\bar{x}_3)=\left(-\mu, \frac{M_2-M_3}{M_1-M_2}-\mu,1-\mu\right),\nn
&(f_2,f_3)=\left(\frac{(M_1-M_2)^3}{(M_2-M_3)(M_1-2M_2+M_3)}, \frac{M_1M_3-M_2^2}{M_1-2M_2+M_3}\right)\!,f_1=1-f_2-f_3,\nn\label{f1f2f3}\\
&\nu^{(3)}=\begin{cases}
\frac{-b+\sqrt{b^2+4ac}}{2a},&a\neq 0,\\
\frac{c}{b},&a=0,
\end{cases}\n
\end{align}
for
\begin{eqnarray}
(a,b,c)=\left(\bar{x}_1\bar{x}_2\bar{x}_3,\,
(M_2-2\mu M_1+\mu^2)+(\bar{x}_1+\bar{x}_2+\bar{x}_3)(\mu-M_1),\,
\mu-M_1\right)\per\n
\end{eqnarray}
\end{theorem}
This theorem is obtained by solving \eqref{bibun0} with $\bF=\bF^{(d)}$ given
in Lemma \ref{lem_solution} below.
\begin{lemma}\label{lem_solution}
If $\mfm{\mo{d}}$ is nonempty then
the solution $\bar{F}^{(d)}$of \eqref{thm_ifonly} is expressed for $d=1,2,3$ as
\begin{eqnarray}
\bar{F}^{(1)}&=&(1-M_1)\mass{0}+M_1\mass{1},\nn
\bar{F}^{(2)}&=&\begin{cases}
\mass{1}&M_1=M_2=1,\\
\frac{(1-M_1)^2}{1-2M_1+M_2}\mass{\frac{M_1-M_2}{1-M_1}}
+\frac{M_2-M_1^2}{1-2M_1+M_2}\mass{1}&\mbox{\rm otherwise},
\end{cases}\nn
\bar{F}^{(3)}&=&\begin{cases}
\bar{F}^{(1)}&M_1=M_2=M_3,\\
f_1\mass{0}+f_2 \mass{\frac{M_2-M_3}{M_1-M_2}}+
f_3\mass{1}
&\mbox{\rm otherwise},
\end{cases}\n
\end{eqnarray}
where
$\mass{x}$ denotes the delta measure at $x$ and $(f_1,f_2,f_3)$ is given by
\eqref{f1f2f3}.
\end{lemma}
This lemma can be confirmed easily
by substitution of $\bar{F}^{(d)}\,(d=1,2,3)$ into \eqref{thm_ifonly}.


\section{Proofs}\label{sec_proofs}
In this section we show Theorems \ref{thm_approach} and \ref{thm_dminm}.
\subsection{A Proof of Theorem \ref{thm_approach}}
Theorem \ref{thm_approach} is proved by a basic result on weak convergence
and L\'evy distance (see, e.g., \cite{lamperti}).
We say that a sequence of probability distributions $\{F_i\}$ converges weakly to $F$ if
$\lim_{i\to\infty}\E_{F_i}[u(X)]=\E_F[u(X)]$
for all bounded, continuous function $u(x)$.
Define the L\'evy distance $L(\cdot,\cdot)$ as
\begin{eqnarray}
L(F,G)=
 \inf\{h>0:\forall x,\,F(x-h)-h\le G(x)\le F(x+h)+h\}\n
\end{eqnarray}
where $F(\cdot)$ and $G(\cdot)$ denote cumulative distribution functions.
A weak convergence is equivalent to the convergence of the L\'evy distance, that is,
$\{F_i\}$ converges weakly to $F$ if and only if $\lim_{i\to\infty}L(F_i,F)=0$.
\begin{proposition}[{\cite[Theorem 7]{honda_colt}}]\label{thm_conti}
$\dmin(F,\mu)$ is continuous in $F\in\mf$ with respect to the L\'evy distance.
\end{proposition}
Now we show Theorem \ref{thm_approach} by Prop.\,\ref{thm_conti}.
\begin{proof}[of Theorem \ref{thm_approach}]
From the continuity of $\dmin(F,\mu)$ in $F$, it suffices to show
for $\dmo{d}=\omo{d}{F}$ that
\begin{eqnarray}
\limsup_{d\to\infty}\sup_{G\in \mfm{\dmo{d}}}L(G,F)=0\per\label{conv_suffice}
\end{eqnarray}

Let $\{G_d\in\mfm{\dmo{d}}\}_{d=1,2,\cdots}$ be a sequence such that
\begin{eqnarray}
\limsup_{d\to\infty}\sup_{G\in \mfm{\dmo{d}}}L(F_d,F)
&=&
\limsup_{d\to\infty}L(G_d,F)=:\bar{L}\per\n
\end{eqnarray}
Since $\mf\supset \mfm{\dmo{d}}$ is compact with respect to the L\'evy distance,
there exist $\bG\in\mf$ and a
convergent subsequence
$\{G_{d_i}\}$ of $\{G_{d}\}$ such that 
\begin{align}
\lim_{i\to\infty}L(G_{d_i},F)&=\bar{L}\com\label{barl}\displaybreak[1]\\
\lim_{i\to\infty}L(G_{d_i},\bG)&=0\com\label{conv_tilde}
\end{align}
where \eqref{conv_tilde} means that $\{G_{d_i}\}$ converges weakly to $\bG$.
From the definition of weak convergence, for all natural numbers $m\in\bbbn$ it holds that
$\lim_{i\to\infty}\E_{G_{d_i}}[X^m]=\E_{\bG}[X^m]$.
On the other hand, $\E_{G_{d_i}}[X^m]=\E_F[X^m]$ for all $d_i\ge m$ from
$G_{d_i}\in \mfm{\dmo{d_i}}$.
Therefore we obtain for all $m\in\bbbn$ that
\begin{eqnarray}
\E_F[X^m]=\lim_{i\to\infty}\E_{G_{d_i}}[X^m]=\E_{\bG}[X^m]\per\n
\end{eqnarray}

Note that a sequence of moments $\{\E_F[X^m]\}$ has one-to-one correspondence to
a distribution $F$ for the case of bounded support.
Therefore $\bG=F$ and
we obtain $\bar{L}=0$ from \eqref{barl}.
\qed
\end{proof}

\subsection{A Proof of Theorem \ref{thm_dminm}}
Theorem \ref{thm_dminm} is proved by theories of Tchebysheff systems and moment spaces
(see Appendix), and the basic result on a saddle-point in the following.
For a function $\varphi(x,y):\mathcal{X}\times \mathcal{Y}\to [-\infty,+\infty]$,
a point $(\bar{x},\bar{y})\in \mathcal{X}\times \mathcal{Y}$ is called a saddle-point
if $\varphi(\bar{x},y)\le \varphi(\bar{x},\bar{y}) \le\varphi(x,\bar{y})$ for all $x\in \mathcal{X}$
and $y\in\mathcal{Y}$.
A necessary and sufficient condition for a saddle-point is
\begin{eqnarray}
\sup_{y\in\mathcal{Y}}\varphi(\bar{x},y)=
\inf_{x\in\mathcal{X}}\sup_{y\in\mathcal{Y}}\varphi(x,y)=
\sup_{y\in\mathcal{Y}}\inf_{x\in\mathcal{X}}\varphi(x,y)=
\inf_{x\in\mathcal{X}}\varphi(x,\bar{y})\per\n
\end{eqnarray}
\begin{proposition}[Minimax Theorem \cite{neumann}]
Let $\mathcal{X}$ and $\mathcal{Y}$ be a compact subset of a topological vector space
$\mathcal{V}$ and $\mathcal{U}$.
Let $\varphi(x,y): \mathcal{X}\times \mathcal{Y}\to [-\infty,+\infty]$ be a function such that
$\varphi(\cdot,y)$ is convex and lower-semicontinuous for any fixed $x$
and $\varphi(x,\cdot)$ is concave and upper-semicontinuous for any fixed $y$.
Then there exists a saddle point $(\bar{x},\bar{y})\in \mathcal{X}\times \mathcal{Y}$.
\end{proposition}

In the proof of Theorem \ref{thm_dminm},
we regard a probability measure $F$ as an element of the family $\vs$ of positive measures
on $[0,1]$
to exploit the results in Appendix.
By letting $\moo{d}:=(1,M_1,\cdots,M_d)$ for $\mo{d}=(M_1,\cdots,M_d)$,
$\dminm{d}(\mo{d},\mu)$ is rewritten as
\begin{eqnarray}
\dmin(\mo{d},\mu)=\inf_{F\in\vs(\moo{d})}\max_{0\le\nu\le \mom}\E_F[\log(1-(X-\mu)\nu)]\com
\label{dminm_henkei}
\end{eqnarray}
where $\vs(\moo{d})$ is the set of positive measures with $0,1,\cdots,d$-th moments equal to $\moo{d}$,
written in \eqref{vs_def}.

\begin{proof}[of Theorem \ref{thm_dminm}]
Let $\mm_{d+1}$ be the moment space with respect to the system $(1,x,\cdots,x^d)$.
It is easily checked that the solutions of \eqref{thm_ifonly} have one-to-one correspondence to
the representations of $\moo{d}\in\mm_{d+1}$ with index at most $d/2$ or
to the upper principal representation.

First consider the trivial case that $\moo{d}$ is a boundary point of $\mm_{d+1}$.
For this case, the proof is straightforward since
$\vs(\moo{d})$ has a single element and its index is at most $d/2$
from Prop.\,\ref{boundary}.

Now we consider the remaining case that $\moo{d}$ is an interior point of $\mm_{d+1}$.
For this case, the upper principal representation of $\mo{d}$
is the unique solution of \eqref{thm_ifonly}
since existence of a representation with index at most $d/2$
implies that $\moo{d}$ is a boundary point of $\mm_{d+1}$ from Prop.\,\ref{boundary}.
In the following, we complete the proof by showing that
the upper principal representation of $\moo{d}$ is the optimal solution $\bar{F}$ in \eqref{thm_uniq}.

Consider applying the minimax theorem to
\eqref{dminm_henkei}.
First, $\mf\supset \vs(\moo{d})$ is compact with respect to the L\'evy distance
and $\E_F[\log(1-(X-\mu)\nu)]$ is linear in $F\in \vs$
for any fixed $\nu$.
Next, $\E_F[\log (1-(X-\mu)\nu)]$ is upper-semicontinuous and concave
in $\nu$ for any fixed $F$.
Then we obtain from the the minimax theorem
that there exists $\bar{\nu}$ satisfying
\begin{eqnarray}
\dmin(\mo{d},\mu)
&=&
\inf_{F\in \vs(\moo{d})}\E_F[\log(1-(X-\mu)\bar{\nu})]\per\n
\end{eqnarray}

Now we show $\bar{\nu}<\momss$ by contradiction.
Assume $\bar{\nu}=\momss$.
From Prop.\,\ref{interior_involve} with $x^*:=1$,
there exists $F'\in\vs(\moo{d})$ such that $F'(\{1\})>0$.
Therefore, from $\log 0=-\infty$,
\begin{eqnarray}
\inf_{F\in \vs(\mo{d})}\E_F[\log(1-(X-\mu)\bar{\nu})]
\le
\E_{F'}[\log(1-(X-\mu)\bar{\nu})]=-\infty\per\n
\end{eqnarray}
It contradicts the positivity of $\dminm{d}$ and $\bar{\nu}<\momss$ is obtained.

From Lemma \ref{t_log} with $\ta:=1+\mu\bar{\nu}$ and $\tb:=\bar{\nu}$, 
$(1,x,\cdots,x^d,-\log(1-(x-\mu)\bar{\nu}))$ is a $T$-system on $[0,1]$
since $\ta/\tb=1/\bar{\nu}+\mu>1$.
Therefore, from Prop.\,\ref{karlin_main}, we obtain
\begin{eqnarray}
\inf_{F\in \vs(\mo{d})}\E_F[\log(1-(X-\mu)\bar{\nu})]
&=&
-\sup_{F\in\vs(\mo{d})}\E_F[-\log(1-(X-\mu)\bar{\nu})]\nn
&=&
-\E_{\bar{F}}[-\log(1-(X-\mu)\bar{\nu})]\nn
&=&
\E_{\bar{F}}[\log(1-(X-\mu)\bar{\nu})]=\dmin(\bF,\mu)\com\n
\end{eqnarray}
where $\bar{F}$ corresponds to the upper principal representation of $\moo{d}$.
\qed
\end{proof}

\section{Improvement of DMED-M Policy}
\label{section_gap}
In DMED-M, $\dmin(F,\mu)$
is bounded from below by $\dminm{d}(\omo{d}{F},\mu)$.
When the gap between $\dminm{d}$ and $\dmin$ is small, DMED-M behaves like
the asymptotically optimal policy, DMED.
%
In this section, we propose DMED-MM policy which is obtained
by a slight modification to DMED-M.
We discuss that DMED-MM works successfully for the case
where the gap between $\dminm{d}$ and $\dmin$ is large.

Define a function $\dmint{d}(F,\mu)$ by
\begin{eqnarray}
\dmint{d}(F,\mu)=\begin{cases}
\dmin(F,\mu)&\E_F\left[\frac{1}{1-X}\right]\le \frac{1}{1-\mu},\\
\dminm{d}(\omo{d}{F},\mu) &\mathrm{otherwise},
\end{cases}\n
\end{eqnarray}
where
recall that
$\dmin(F,\mu)=\E_F[\log(1-X)]-\log(1-\mu)$ for the first case.
{\it DMED-MM (DMED-M Mixed)} policy is obtained by replacing
$\dminm{d}(\omo{d}{\fhatn{i}{n}},\allowbreak\mu)$
in DMED-M by $\dmint{d}(\fhatn{i}{n},\mu)$.
Then the criterion for choosing an arm is the same as DMED
for the case $\E_{\fhatn{i}{n}}[1/(1-X)]\le 1/(1-\mu)$ and
the same as DMED-M otherwise.

In the first place, $\dminm{d}(\omo{d}{\fhatn{i}{n}},\mu)$ in DMED-M is easy to compute since
the empirical moments
$\omo{d}{\fhatn{i}{n}}$
can be obtained in constant time from the sum
$\sum_{t=1}^{T_i(n)} X_{i,t}^m$ where $X_{i,t}$ denotes the $t$-th
reward from arm $i$.
On the other hand, $\dmin(\fhatn{i}{n},\mu)\allowbreak =\max_{\nu}\E_{\fhatn{i}{n}}[1-(X-\mu)\nu]$
in DMED requires
the computation of the sum
$\sum_{t} \log(1-(X_{i,t}-\mu)\nu)$
where $\mu$ and $\nu$
generally take various values.
In this viewpoint, the computation of $\dmint{d}(\fhatn{i}{n},\mu)$ is practical
since it is obtained from the sums $\sum_t X_{i,t}^m$,\,
$\sum_t 1/(1-X_{i,t})$ and $\sum_t \log(1-X_{i,t})$.

Now consider the maximum gap between $\dminm{d}(\omo{d}{F},\mu)$ and $\dmin(F,\mu)$
among distributions $F$ with moment $\mo{d}$.
\begin{lemma}\label{thm_sup}
The supremum in
\begin{eqnarray}
\!\!\!\!\!\!\!\!\!\!
\sup_{F\in \mfm{\mo{d}}}\left(\dmin(F,\mu)-\dminm{d}(\mo{d},\mu)\right)
=
\sup_{F\in \mfm{\mo{d}}}\dmin(F,\mu)-\dminm{d}(\mo{d},\mu)
\label{dmin_sup}
\end{eqnarray}
is attained by the unique solution of
\begin{eqnarray}
\begin{cases}
\sum_{i=1}^{\lceil (d+1)/2\rceil} f_i x_i^m=M_m\;(m=0,\cdots,d),&\mbox{\rm $d$ is odd},\\
\sum_{i=1}^{\lceil (d+1)/2\rceil} f_i x_i^m=M_m\;(m=0,\cdots,d),\, x_{1}=0, &\mbox{\rm $d$ is even}.
\end{cases}\label{lem_ifonly}
\end{eqnarray}
\end{lemma}
The proof of the lemma is parallel to that of Theorem \ref{thm_dminm}
which considers the infimum of $\dmin(F,\mu)$.
In Theorem \ref{thm_dminm},
we saw that the upper principal representation $\bF$ of $\mo{d}$ attains the infimum
from Prop.\,\ref{karlin_main}.
Similarly, we can show that the lower principal representation $\uF$ of $\mo{d}$ attains the supremum
in \eqref{dmin_sup}.

Note that we can obtain an explicit expression of \eqref{dmin_sup} for small $d$
in the same way as Theorem \ref{thm_d123}.
However, such an expression is a complicated function on $\mo{d}$ as Theorem \ref{thm_d123}
and it is not useful as an evaluation of the gap, since we cannot know the value of the expression
until we substitute the specific value of $\mo{d}$.

Lemma \ref{thm_sup} is useful when we consider the performance of DMED-MM.
Compare the solution $\uF$ of \eqref{lem_ifonly}
to the upper principal representation $\bF$ in \eqref{thm_ifonly}.
For odd $d$, $\uF$ is supported by fewer points which generally contain neither $0$ nor $1$.
For even $d$, $\uF$ is supported by the same number of points which contain $0$ instead of $1$.
In any case, we can say qualitatively that $\uF$, distribution such that
DMED-M behaves badly (i.e., gap between $\dmin$ and $\dminm{d}$ is large),
has small weight around $1$.

Note that such a distribution often satisfies $\E_F[1/(1-X)]\le 1/(1-\mu)$
since $\E_F[1/(1-X)]$ is controlled mainly by the weight around $1$.
In fact, we can show from Prop.\,\ref{karlin_main} that $\min_{F\in\mfm{\mo{d}}}\E_F[1/(1-X)]$ is attained
by the lower principal representation $\uF$,
which also attains the supremum in \eqref{dmin_sup}.



Now we summarize the above argument:
(1) DMED-M behaves most differently from DMED for $\uF$
among distributions with the moment $\mo{d}$.
(2) Among these distributions $F$, $\uF$ is also the distribution minimizing
$\E_F[1/(1-X)]$,
although the minimum value is not always smaller than $1/(1-\mu)$.
(3) If $\E_F[1/(1-X)]\le 1/(1-\mu)$, DMED-MM behaves in the same way as DMED
(otherwise it does in the same way as DMED-M).
In this sense, the worst case gap between DMED-MM and DMED
is sometimes smaller than that between DMED-M and DMED.

\section{Experiments}\label{sect_sim}
In this section we show some numerical results on DMED-MM and the function $\dminm{d}$.

First, we compare the performance of $1,2,3$-th degree DMED-MM
and DMED in Fig.\,\ref{fig1}.
\begin{figure}[tb]
 \begin{center}
  \includegraphics[bb=135 220 490 630,clip,angle=270,width=69mm]{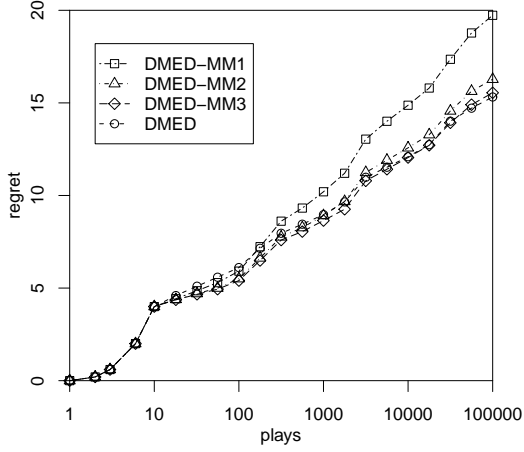}
 \end{center}
 \caption{Empirical regrets of $1,2,3$-th degree DMED-MM and DMED for beta distributions.
 Each plot is an average over 1000 different runs.}
 \label{fig1}
\end{figure}
Each plot is an average over 1000 different runs
and we used $5$ arms with beta distributions $\mathrm{Be}(\alpha,\beta)$.
Note that beta distribution covers various forms of distributions on $[0,1]$.
The parameters of the arms are
$(\alpha,\beta)=(9,1),(0.7,0.3),\allowbreak(5,5),(0.3,0.7),(1,9)$
with expectations $\mu=0.9,\,0.7,\,0.5,\,0.3,\,0.1$. 
The vertical axis denotes the regret  $\sum_{i:\mu_i<\mu^*}(\mu^*-\mu_i)T_i(n)$, which is
the loss due to choosing suboptimal arms.
We see from the figure that the performance of DMED-MM approaches DMED
as the degree increases.

Next, we show values of $\dminm{d}$ and whether $\E_F[1/(1-X)]\le 1/(1-\mu)$ or not
for various distributions in Table \ref{table1}.
Recall that DMED-MM works the same as DMED for the case $\E_F[1/(1-X)]\le 1/(1-\mu)$.
\begin{table}[tb] 
\caption{Values of $\dminm{d}(\omo{d}{F},\mu)$ and $\dmin(F,\mu)$ for beta distributions.
}
\label{table1}
\begin{center}
{\normalsize
{\renewcommand\arraystretch{1.2}
\begin{tabular}{|c|c|c||c|c|c|c|c|}
\hline
$F$&$\,\E_F[X]\,$&$\mu$&$\dminm{1}$&$\dminm{2}$&$\dminm{3}$&$\dmin$
&$\,\E_F\big[\frac{1}{1-X}\big]\le\frac{1}{1-\mu}\!\phantom{\Big|}$\\
\hline
$\mathrm{Be}(2,2)$&0.5&0.6&0.0204&0.0703&0.0843&0.0984&False\\
\hline
$\mathrm{Be}(0.5,0.5)$&0.5&0.6&0.0204&0.0366&0.0400&0.0408&False\\
\hline
$\mathrm{Be}(1,3)$&0.25&0.6&0.253&0.459&0.522&0.583&True\\
\hline
$\mathrm{Be}(0.25,0.75)$&0.25&0.6&0.253&0.348&0.391&0.431&False\\
\hline
$\mathrm{Be}(2,2)/2$&0.25&0.3&0.00617&0.0373&0.0490&0.0576&True\\
\hline
\,$\mathrm{Be}(0.5,0.5)/2\,$&0.25&\,0.3\,&\,0.00617\,&\,0.0239\,&\,0.0337\,&\,0.0401\,&True\\
\hline
\end{tabular}
}
}
\end{center}
\end{table}
Distribution $\mathrm{Be}(\alpha,\beta)/2$ denotes 
the distribution of $X/2$ for random variable $X$ with distribution $\mathrm{Be}(\alpha,\beta)$.
It corresponds to the case that the upper bound of the support of distributions
is unknown and assumed conservatively as $2$ instead of $1$.
For this case, a reward $X$ is passed as $X/2$ to a policy for distributions on $[0,1]$.
We see from the figure that
$\dminm{d}$ bounds $\dmin$ from below accurately when $\E_F[1/(1-X)]\le 1/(1-\mu)$ is false,
as discussed in the previous section.
Overall, the gap between $\dminm{1}$ and $\dminm{2}$ seems to be very large
and it seems to be necessary to use at least the second moment (i.e., variance)
to achieve a smaller regret.

\section{Conclusion}
\label{section-remarks}
In this paper we proposed DMED-M policy which
is computed by the first $d$ empirical moments of the arms.
The regret bound of DMED-M approaches that of DMED, which is asymptotically optimal,
as $d$ increases.
The computation involved in DMED-M is represented in an explicit form for small $d$.
We also proposed DMED-MM policy, which sometimes
improves the worst case performance of DMED-M.

An open problem is whether the asymptotic bound of DMED-M is the best
for all policies which only consider the empirical moments.
We may be able to prove the optimality of DMED-M in this sense under some regularity conditions.



\bibliographystyle{splncs}
\bibliography{bunken}

\appendix
\section{Tchebycheff Systems and Moment Spaces}
In this appendix we summarize results on Tchebycheff systems and moment spaces
needed for the proof of Theorem \ref{thm_dminm}.
All functions and measures are defined on $[a,b]\,(a<b)$ in this appendix
whereas they are on $[0,1]$ elsewhere.
For any set of points $\{x_1,\cdots,x_l\}$, we always assume
$a\le x_1<x_2<\cdots<x_l\le b$.

 \begin{definition}
Let $u_0(x),\cdots,u_d(x)$ denote continuous real-valued functions on $[a,b]$.
These functions are called a Tchebycheff system (or T-system) if
determinants
\begin{eqnarray}
\mathrm{det}
\left(\begin{array}{cccc}
u_0(x_0)& u_1(x_0) &\cdots&u_d(x_d)\\
u_0(x_1)& u_1(x_1) &\cdots&u_d(x_d)\\
\vdots& \vdots &&\vdots\\
u_0(x_d)& u_1(x_d) &\cdots&u_d(x_d)\\
\end{array}\right)\label{determinant}
\end{eqnarray}
are
positive for all $\{x_0,\cdots,x_d\}$.
\end{definition}

A typical $T$-system is $u_i(x)=x^i \,(i=0,1,\cdots,d)$, where \eqref{determinant}
is represented as the Vandermonde determinant
\begin{eqnarray}
\mathrm{det}
\left(\begin{array}{cccc}
1&x_0&\cdots&x_0^d\\
1&x_1&\cdots&x_1^d\\
\vdots& \vdots &&\vdots\\
1&x_d&\cdots&x_d^d\\
\end{array}\right)
=\prod_{1\le i< j\le d}(x_j-x_i)>0\per \n
\end{eqnarray}

Let $Z(u)$ of a function $u(x)$ denote the number of distinct points $x\in[a,b]$ such that
$u(x)=0$.
Then $T$-systems are discriminated by the following proposition.
\begin{proposition}[{\cite[Chap. I, Theorem 4.1]{karlin}}]\label{tobeT}
If a system $\{u_i\}_{i=0}^d$ of continuous functions on $[a,b]$ satisfies $Z(u)\le d$
for all
\begin{eqnarray}
u(x)=\sum_{i=0}^d a_i u_i(x),\,\;\{a_i\}\in \bbbr^{d+1}\setminus \{0^{d+1}\}\com \n
\end{eqnarray}
 then
$(u_0,u_1,\cdots,u_{d-1},u_d)$ or $(u_0,u_1,\cdots,u_{d-1},-u_d)$ is a $T$-system.
\end{proposition}

\begin{lemma}\label{t_log}
For any $\ta$ and $\tb>0$ satisfying $b<\ta/\tb$, $(1,x,\cdots,x^d,-\log(\ta-\tb x))$
is a $T$-system on $[a,b]$.
\end{lemma}

\begin{proof}
Let $b'\in (b,\,\ta/\tb)$ be sufficiently close to $\ta/\tb$ and
consider function
\begin{eqnarray}
u(x)=\sum_{m=0}^d a_m x^m -a_{d+1}\log(\ta-\tb x) \n
\end{eqnarray}
on $x\in[a,b']$.
Since the derivative of $u(x)$ is written as
\begin{eqnarray}
\bibun{u(x)}{x}=\frac{(\ta-\tb x)\sum_{m=1}^d a_m x^{m-1}+a_{d+1}q}
{\ta-\tb x}\com \n
\end{eqnarray}
$u(x)$ has at most $d$ extreme points in $[a,b']$.
Therefore $Z(u)\le d+1$ and
$(1,x,\cdots,x^d,\allowbreak \log(\ta-\tb x))$ or
$(1,x,\cdots,x^d,-\log(\ta-\tb x))$ is a $T$-system on $[a,b']$
from Prop.\,\ref{tobeT}.

The determinant \eqref{determinant} for the system
$(1,x,\cdots,x^d,\allowbreak \log(\ta-\tb x))$ is written as
\begin{eqnarray}
\lefteqn{
\mathrm{det}
\left(\begin{array}{lllll}
1\:&x_0\:&\cdots&x_0^d\:&\log(\ta-\tb x_0)\\
1\:&x_1\:&\cdots&x_1^d\:&\log(\ta-\tb x_1)\\
\vdots& \vdots &&\vdots&\qquad\vdots\\
1\:&x_{d+1}\:&\cdots&x_{d+1}^d\:&\log(\ta-\tb x_{d+1})\\
\end{array}\right)
}\nn
&=&
\sum_{m=0}^{d+1} (-1)^{d+m+1}\left(\prod_{0\le i< j\le d+1:i,j\neq m}(x_j-x_i)\right) \log(\ta-\tb x_m)
\per\n
\end{eqnarray}
For the case that $x_{d+1}=b'$ with $b'\uparrow \ta/\tb$,
$\log(\ta-\tb x_{d+1})$ goes to $-\infty$
and the sign of the determinant is controlled by the term involving $\log(\ta-\tb x_{d+1})$,
which is written as
\begin{eqnarray}
(-1)^{2m+2}\left(\prod_{0\le i< j\le d+1:i,j\neq m}(x_j-x_i)\right)\log (\ta-\tb x_{d+1})<0\per\n
\end{eqnarray}
Then, $(1,x,\cdots,x^d,\allowbreak \log(\ta-\tb x))$ cannot be a $T$-system
on $[a,b']$ for $b'$ sufficiently close to $\ta/\tb$ and therefore
$(1,x,\cdots,x^d,\allowbreak -\log(\ta-\tb x))$ has to be a $T$-system on $[a,b']$.
From the definition of $T$-system, it also is a $T$-system on $[a,b]\subset [a,b']$.
\qed
\end{proof}

Let $\vs$ be the family of positive measures on $[a,b]$
and define a subset $\vs(\moap{d})$ of $\vs$
for a vector $\moap{d}=(\moapo_0,\moapo_1,\cdots,\moapo_d)$ as
\begin{eqnarray}
\vs(\moap{d})=\left\{\si\in \vs: \forall m\in\{0,1,\cdots,d\},\,\int_a^b x^m \rd \si(x)=\moapo_m \right\}\per
\label{vs_def}
\end{eqnarray}
The notion of {\it moment spaces} is essential to examine properties of $T$-systems.
\begin{definition}
The moment space $\mm_{d+1}$ with respect to the $T$-system $\{u_i\}$ is given by
\begin{eqnarray}
\lefteqn{
\mm_{d+1}
}\nn
&\equiv& \left\{\left(\int_a^b u_0(x)\rd \sigma(x),\int_a^b u_1(x)\rd \sigma(x),\cdots, \int_a^b u_d(x)\rd \sigma(x)
\right)\in \bbbr^{d+1}
: \sigma\in \vs\right\}\per\n
\end{eqnarray}
\end{definition}

Consider the case that $\moap{d}\in\mm_{d+1}$ satisfies
\begin{eqnarray}
M_m=\sum_{i=1}^l f_i u_m(x_i)\quad (m=0,\cdots,d) \label{rep}
\end{eqnarray}
with $x_1,\cdots,x_l\in [a,b]$ and $f_1,\cdots,f_l>0$ for any finite $l$.
We call such an expression {\it representation} of $\moap{d}$.
A representation of $\moap{d}$ corresponds uniquely to the measure
\begin{eqnarray}
\si=\sum_{i=1}^l f_i\de(x_i)\in\vs\n
\end{eqnarray}
for the delta measure $\de(x)$ at point $x$.
We sometimes identify the measure $\si$ with the representation of $\moap{d}$.
The measure $\si$ is a probability measure if $\sum_{i}f_i=1$.

The {\it index} of the representation $\eqref{rep}$ is defined as the number of the points
$(x_1,\cdots,x_l)$ 
under the special convention that the points $a,b$ are counted as one half.
A representation is called {\it principal} if its index is $(d+1)/2$.
Furthermore, the representation is {\it upper} if $(x_1,\cdots,x_l)$ contains $b$ and {\it lower} otherwise.

For the proof of Theorem \ref{thm_dminm}, it is necessary
to study the nature on the set $\vs(\moap{d})$.
It differs according to whether $\moap{d}$ is a boundary point of $\mm_{d+1}$
or an interior point of $\mm_{d+1}$.
\begin{proposition}[{\cite[Chap. II, Theorem 2.1]{karlin}}]\label{boundary}
Every boundary point $\moap{d}$ has a unique representation.
Moreover, $\moap{d}\in\mm_{d+1}$ is a boundary point of $\mm_{d+1}$ if and only if
there exists a representation of $\moap{d}$ with index at most $d/2$.
\end{proposition}
\begin{proposition}[{\cite[Chap. II, Theorem 3.1]{karlin}}]\label{interior_involve}
If $\moap{d}$ is an interior point of $\mm_{d+1}$ then, for arbitrary $x^*\in [a,b]$,
there exists a representation of $\moap{d}$ such that
$(x_1,\cdots,x_l)$ contains $x^*$.
\end{proposition}
\begin{proposition}[{\cite[Chap. II, Corollary 3.1]{karlin}}]\label{karlin_unique}
If $\moap{d}$ is an interior point of $\mm_{d+1}$ then
there exist precisely one upper and one lower principal representations of $\moap{d}$.
\end{proposition}
We use Prop.\,\ref{karlin_unique}
implicitly in Prop.\,\ref{karlin_main} below and the proof of Theorem \ref{thm_dminm}.
Prop.\,\ref{karlin_main} is the main result of the appendix.
\begin{proposition}[{\cite[Chap. III, Theorem 1.1]{karlin}}]\label{karlin_main}
Assume $(u_0,u_1,\cdots,u_d)$ and $(u_0,u_1,\cdots,u_d,h)$ are T-systems.
Then
\begin{eqnarray}
\max_{\si\in \vs(\moap{d})}\int_a^b h(x)\rd \si(x)\n
\end{eqnarray}
is attained uniquely by $\bar{\si}$, the upper principal
representation of $\moap{d}$.
Similarly, 
\begin{eqnarray}
\min_{\si\in \vs(\moap{d})}\int_a^b h(x)\rd \si(x)\n
\end{eqnarray}
is attained uniquely by $\underline{\si}$, the lower principal 
representation of
$\moap{d}$.
\end{proposition}


\end{document}